\documentclass[a4paper]{amsart} 
\usepackage{amsmath, amsthm, amsfonts, amssymb}

\newtheorem{theorem}{Theorem}[section]
\newtheorem{lemma}[theorem]{Lemma}
\newtheorem{corollary}[theorem]{Corollary}

\newcommand{\transf}[2]{{T[#1,#2]}}
\newcommand{\UnitDisk}{{\mathbb D}}
\newcommand{\UnitCircle}{{\partial \UnitDisk}}

\newcommand{\wco}{{W_{\psi,\varphi}}}
\newcommand{\e}{{\mathrm e}}

\newcommand{\BD}{\mathbb{D}}

\newcommand{\ra}{\rightarrow}

\begin{document}

\title{The essential norm of a weighted composition operator on BMOA}

\author{Jussi Laitila and Mikael Lindstr\"{o}m}
\address{\noindent Department of Biosciences, University of Helsinki, P.O. Box 65, FI-00014 University of Helsinki, Finland}
\email{jussi.laitila@helsinki.fi}
\address{\noindent Department of Mathematical Sciences, University of Oulu, P.O. Box 3000,  FI-90014 University of Oulu, Finland}
\email{mikael.lindstrom@oulu.fi}

\maketitle

\begin{abstract}
We provide an estimate for the essential norm of a weighted composition operator $\wco\colon f\mapsto \psi(f\circ\varphi)$ acting on the space $BMOA$ in terms of the weight function $\psi$ and the $n$-th power $\varphi^n$ of the analytic self-map $\varphi$ of the open unit disc $\BD$. We also provide a new estimate for the norm of the weighted composition operator on $BMOA$.
\end{abstract}

\section{Introduction}

Let $\BD$ be the open unit disc in the complex plane  and denote by $H(\BD)$ the space of all analytic functions $\UnitDisk \to \mathbb C$. Throughout the paper, $\psi$ will denote a function in $H(\BD)$, and $\varphi$
will be an analytic self-map of $\BD$, $\varphi(\UnitDisk)\subset \UnitDisk$. These maps induce a linear weighted composition operator 
$W_{\psi,\varphi}$  which is defined on $H(\BD)$ by 
\begin{align*}
(\wco f)(z)=(M_\psi C_\varphi f)(z)=\psi(z)f(\varphi(z)), \quad z\in\BD,
\end{align*}
where $M_\psi$ is the operator of pointwise multiplication by $\psi$, and $C_\varphi$ is the composition operator 
$f\mapsto f\circ\varphi$. The main aim of this paper is to complement the literature on weighted composition operators by providing a function-theoretic estimate for the essential norm of a weighted composition operator $\wco$ on the space $BMOA$, which consists of the analytic functions on $\BD$ that are of bounded mean oscillation on the unit circle $\UnitCircle$. 

Let $H^p$  ($1\le p < \infty$) be the classical Hardy space of functions $f\in H(\BD)$ that satisfy
$$ \Vert f \Vert_{p} = \lim_{r\to 1} \left(\int_0^{2\pi} |f(r\e^{i\theta})|^p \,\frac{\mathrm d\theta}{2\pi}\right)^{1/p}$$
and let $H^\infty$ by the space of bounded analytic functions on $\UnitDisk$ endowed with the sup norm. 
Then $f\in H^2$ belongs to the space $BMOA$ provided 
that the $BMOA$ seminorm 
$$\Vert f\Vert_{\ast}=\sup_{a \in \BD} \Vert \transf{f}{a} \Vert_2$$
is finite. 
Here $\transf{f}{a} = f\circ\sigma_a - f(a)$, where $\sigma_a(z) = (a -z)/(1 - \overline a z)$ is the conformal automorphism of $\BD$ that 
exchanges the points $0$ and $a$.
The quantity $\Vert f \Vert = |f(0)| + \Vert f \Vert_{\ast}$ is a complete norm on $BMOA$. 
The closed subspace $VMOA$ consists of the analytic functions having vanishing mean oscillation on $\partial\BD$, or equivalently, of the functions $f\in BMOA$ such that 
$$\lim_{|a|\to 1}  \Vert \transf{f}{a} \Vert_2 = 0.$$
We refer to \cite{gar}, \cite{gir} and \cite{zhu} for the basic properties of the spaces $BMOA$ and $VMOA$.

Boundedness and compactness properties of pointwise multipliers $M_\psi$ and composition operators $C_\varphi$ on $BMOA$ and $VMOA$ have been intensively investigated
in the literature; see, for example, \cite{steg}, \cite{smith}, \cite{lai1}, \cite{wulan1}, \cite{lai2}, \cite{wulan}, \cite{LNST}. 
In \cite{lai2}, function-theoretic characterizations were given for the bounded and compact weighted composition operators on both $BMOA$ and $VMOA$. In addition, the following estimates were given for the norm of $\wco$ on $BMOA$ and its essential norm on the subspace $VMOA$:
\begin{align}\label{eq_norm_bmoa}
\Vert \wco \Vert_{L(BMOA)} \asymp |\psi(0)|L(\varphi(0)) + \sup_{a\in\UnitDisk} \alpha(a) + \sup_{a\in\UnitDisk}\beta(a);
\end{align}
\begin{align}\label{eq_ess_norm_vmoa}
\Vert \wco \Vert_{e,VMOA} \asymp\limsup_{|a|\to 1} \alpha(a) + \limsup_{|a|\to 1}\beta(a).
\end{align}
Here we have used the abbreviations
$$\alpha(a)=|\psi(a)| \Vert \sigma_{\varphi(a)}\circ\varphi\circ\sigma_a \Vert_2  \quad \mathrm{and}   \quad  \beta(a)=L(\varphi(a))\Vert \psi\circ\sigma_a - \psi(a)\Vert_2,$$
where 
$L(a)=\log\frac{2}{1-|a|^2}$.
However, the estimate \eqref{eq_ess_norm_vmoa} does not in general extend to $BMOA$, except for symbols $\psi$ and $\varphi$ for which $\wco$ preserves $VMOA$.

Recently, new types of compactness characterizations and essential norm estimates have been given on many function spaces 
in terms of the $n$-th powers $\varphi^n$ of the symbol $\varphi$; see \cite{zhao}, \cite {manhas}, \cite{mik1} and \cite{mik2}.
Continuing this line of research, Colonna \cite{col} provided the following modifications of results from \cite{lai2}, which in particular simplify the
compactness characterization on $BMOA$: 
The operator $\wco$ is bounded on $BMOA$ if and only if
$\sup_{n\ge 0} \Vert \psi \varphi^n \Vert_\ast<\infty$ and $\sup_{a\in\UnitDisk} \beta(a)<\infty$, and it is compact on $BMOA$ if and only if
\begin{align*}
\limsup_{n\to \infty} \Vert \psi \varphi^n \Vert_\ast =0 \quad \textrm{ and } \quad \limsup_{|\varphi(a)|\to 1}\beta(a) = 0.
\end{align*}
Very recently Pablo Galindo and the authors of this paper \cite{GLL} provided an estimate for the essential norm of a composition operator on $BMOA$ in terms of the $BMOA$ norm of $\varphi^n$.
This motivates the natural question of whether such an estimate could be extended for weighted composition operators. We shall give a positive answer to this question by extending the approach
from \cite{GLL} and combining it with earlier works \cite{lai2} and \cite{col} on weighted composition operators.  

Recall that the essential norm $\|T\|_{e,X}$ of a bounded operator $T\colon X\ra X$ is defined as the distance from $T$
to the space of compact operators on $X$. For two quantities $A$ and $B$ which may depend on $\varphi$ and $\psi$, we use the 
abbreviation $A \lesssim B$ whenever there is a positive constant $c$ (independent of $\varphi$ and $\psi$) such that $A \le c B$. We write $A \asymp B$, if $A \lesssim B \lesssim A$.
For basic results about composition operators on classical spaces of analytic functions, we refer the  reader to the monographs \cite{shapiro1} and \cite{comac}. 

We next formulate our main results. Theorem \ref{mainthm1} provides a new estimate for the norm of a weighted composition operator on $BMOA$.
Theorem \ref{mainthm2} provides an estimate for the essential norm $\Vert \wco \Vert_{e,BMOA}$.

\begin{theorem}\label{mainthm1} Suppose that $\wco$ is bounded on $BMOA$. Then
\begin{align*}
\Vert \wco \Vert_{L(BMOA)} \asymp |\psi(0)|L(\varphi(0)) + \sup_n \Vert \psi \varphi^n \Vert_\ast + \sup_{a\in\UnitDisk}\beta(a).
\end{align*}
\end{theorem}

\begin{theorem}\label{mainthm2} Suppose that $\wco\colon BMOA \to BMOA$ is bounded. Then
\begin{align*}
\Vert \wco \Vert_{e,BMOA} \asymp \limsup_{n\to\infty}  \Vert \psi \varphi^n \Vert_{\ast} + \limsup_{|\varphi(a)|\to 1}\beta(a).
\end{align*}
\end{theorem}

We next formulate two corollaries to these results. The first one is a direct consequence of Theorem \ref{mainthm2}. Here the estimate for pointwise multipliers appears to be new; the estimate for composition operators was earlier given in \cite{GLL}. 

\begin{corollary} Suppose that $M_\psi$ is bounded on BMOA. Then
\begin{align*}
\Vert M_\psi \Vert_{e,BMOA} \asymp  \limsup_{n\to\infty}  \Vert \psi z^n \Vert_{\ast} + \limsup_{|a|\to 1}L(a)\Vert \transf{\psi}{a}\Vert_2
\end{align*}
and 
\begin{align*}
\Vert C_\varphi \Vert_{e,BMOA} \asymp \limsup_{n\to\infty}  \Vert \varphi^n \Vert_{\ast}.
\end{align*}
\end{corollary}

The following result is a variant of formula \eqref{eq_ess_norm_vmoa}, and it is proved in Section \ref{sec_pfofthm2}.

\begin{corollary}\label{corend} If $\wco$ is bounded on $VMOA$, then
\begin{align*}
\Vert \wco \Vert_{e,BMOA}\asymp \Vert \wco \Vert_{e,VMOA} \asymp \limsup_{|\varphi(a)|\to 1} \alpha(a) + \limsup_{|\varphi(a)|\to 1}\beta(a).
\end{align*}
\end{corollary}

The rest of the paper is organized as follows. Section \ref{sec_backgd} contains preparatory material. In Sections \ref{sec_pfofthm1} and \ref{sec_pfofthm2} we establish the proofs of Theorems \ref{mainthm1} and \ref{mainthm2}.


\section{Preliminary results}\label{sec_backgd}

In this section we collect several preliminary results needed for the proofs of our main theorems. 
We start with known properties of $BMOA$ functions.
First, we need a pointwise estimate, which follows by integrating the well-known inequality $|f'(z)|\lesssim \Vert f \Vert(1-|z|^2)^{-1}$, which holds for  all $f\in BMOA$; see, for example, \cite{gir} or \cite{zhu}. 

\begin{lemma}\label{lemmaptwiseie}
For $f\in BMOA$ and $z\in \UnitDisk$,
\begin{align*}
|f(z)|\lesssim L(z)\Vert f \Vert,
\end{align*}
where $L(z)=\log\frac{2}{1-|z|^2}$.
\end{lemma}

For $1 \le p < \infty$, we define the $BMOA$ $p$-seminorm by 
$\Vert f \Vert_{\ast,p}=\sup_a \Vert \transf{f}{a} \Vert_p$.  The following reverse H\"older inequality is a consequence of the John-Nirenberg lemma; see \cite{gar}, \cite{gir} or \cite{zhu}.
 
\begin{lemma}\label{lemmajn}
For every $1 \le p < \infty$,
\begin{align*}
\Vert f\Vert_{\ast} \asymp \Vert f \Vert_{\ast,p}
\end{align*}
\end{lemma}

The next lemma collects two inequalities which were proved in \cite[Prop.~2.3 and (3.19)]{lai2}. The first one of them is a version of the Littlewood inequality (\cite{shapiro1}, \cite{comac}). It implies, in particular, the well-known result that all composition operators are bounded on $BMOA$; see, for example, \cite{GLL}.

\begin{lemma}[{\cite{lai2}}]\label{lemmaweightedlwie}
Let $g\in H^2$ be such that $g(0)=0$. Then for $\varphi$ such that  $\varphi(0)=0$, we have
\begin{align*}
\Vert g \circ \varphi \Vert_2 \lesssim \Vert \varphi \Vert_2 \Vert g \Vert_2
\end{align*}
and for $\frac{1}{2}\le t < 1$ and $|z|\le t$, we have
\begin{align*}
|g(z)| \le 2|z|\max_{|w|\le t}|g(w)|.
\end{align*}
\end{lemma}


The next result contains four basic estimates about weighted composition operators which will be applied a number of times in the sequel. 
These estimates were essentially proved in \cite{lai2} (Proof of Theorem 2.1 and Lemma 3.4), but were not collected there as independent results.
We also formulate these estimates slightly differently from \cite{lai2} and hence present short proofs.

\begin{lemma}[\cite{lai2}]\label{lemmaupperest}

(i) Let $f_a = \sigma_{\varphi(a)}-\varphi(a)$. For all $a\in\UnitDisk$,
\begin{align*}
\alpha(a) \lesssim \beta(a)/L(\varphi(a)) + \Vert \wco f_a \Vert_{\ast}.
\end{align*}

(ii) Let $g_a = h_a^2/h_a(\varphi(a))$, where $h(z)=\log(2/(1-\overline{\varphi(a)}z)$. For all $a\in\UnitDisk$,
\begin{align*}
\beta(a) \lesssim \Vert \transf{\psi}{a} \transf{g_a\circ\varphi}{a}\Vert_2 + \Vert\wco g_a \Vert_{\ast} + \alpha(a).
\end{align*}

(iii) For all $f\in BMOA$ and $a\in\UnitDisk$,
\begin{align*}
\Vert \transf{\wco f}{a} \Vert_2  \lesssim \Vert \transf{\psi}{a} \transf{f\circ\varphi}{a}\Vert_2+ \left(\alpha(a) + \beta(a) \right)\Vert f \Vert_{\ast}.
\end{align*}

(iv) For all $f\in BMOA$ and $a\in\UnitDisk$,
\begin{align*}
\Vert \transf{\psi}{a} \transf{f\circ\varphi}{a}\Vert_2 \lesssim \Vert f \Vert_{\ast} \min\left\{\sup_{w\in\BD}\beta(w), \frac{\Vert \wco \Vert_{L(BMOA)}}{\sqrt{L(\varphi(a))}}\right\}.
\end{align*}
\end{lemma}

\begin{proof} (i) Since $\Vert f_a \circ\varphi\circ\sigma_a \Vert_\infty \le 2$, we have
\begin{align*}
\alpha(a)&=|\psi(a)| \Vert \transf{f_a\circ\varphi}{a} \Vert_2 \\
&= \Vert \transf{\psi}{a}\cdot f_a \circ\varphi\circ\sigma_a - \transf{\wco f_a}{a}\Vert_2\\
&\le 2\Vert \transf{\psi}{a} \Vert_2 + \Vert \wco f_a\Vert_{\ast}\\
&=2 \beta(a)/L(\varphi(a)) + \Vert \wco f_a\Vert_{\ast}.
\end{align*}

(ii)  By the triangle inequality and because $g_a(\varphi(a)))=L(\varphi(a))$, we get
\begin{align*}
\beta(a) &= \Vert g_a (\varphi(a))\transf{\psi}{a} \Vert_2 \\
&= \Vert \transf{g_a\circ\varphi}{a}  \transf{\psi}{a} + \psi(a)\transf{g_a\circ\varphi}{a} - \transf{\wco g_a}{a} \Vert_2 \\
&\le \Vert \transf{g_a\circ\varphi}{a}  \transf{\psi}{a} \Vert_2 + |\psi(a)|\Vert \transf{g_a\circ\varphi}{a} \Vert_2 + \Vert \wco g_a \Vert_{\ast}.
\end{align*}
Note that $\transf{g_a\circ\varphi}{a} = g_a\circ\sigma_{\varphi(a)}\circ(\sigma_{\varphi(a)}\circ\varphi\circ\sigma_a)-g_a(\varphi(a))$, where 
$(g_a\circ\sigma_{\varphi(a)}-g_a(\varphi(a)))(0) = (\sigma_{\varphi(a)}\circ\varphi\circ\sigma_a)(0) = 0$. Hence
by Lemma \ref{lemmaweightedlwie}, we have
\begin{align*}
|\psi(a)|\Vert \transf{g_a\circ\varphi}{a} \Vert_2 \lesssim \alpha(a) \Vert g_a \Vert_{\ast}.
\end{align*}
Because $\sup_{a\in\BD}\Vert g_a \Vert_{\ast} < \infty$, this yields (ii).

(iii) Again, by the triangle inequality,
\begin{align*}
\Vert \transf{\wco f}{a} \Vert_2 
&= \Vert 
\transf{\psi}{a}\transf{f\circ\varphi}{a} + \psi(a)\transf{f\circ\varphi}{a} +  \transf{\psi}{a}f(\varphi(a)) \Vert_2 \\
&\le 
 \Vert \transf{\psi}{a}\transf{f\circ\varphi}{a} \Vert_2 +
|\psi(a)|\Vert \transf{f\circ\varphi}{a} \Vert_2 + |f(\varphi(a))|\Vert \transf{\psi}{a}\Vert_2,
\end{align*}
so, by Lemmas \ref{lemmaweightedlwie} and \ref{lemmaptwiseie},
\begin{align*}
|\psi(a)|\Vert \transf{f\circ\varphi}{a} \Vert_2 + |f(\varphi(a))|\Vert \transf{\psi}{a}\Vert_2 
&\lesssim
\alpha(a)\Vert f \Vert_{\ast}+ L(\varphi(a))\Vert \transf{\psi}{a}\Vert_2\Vert f \Vert_{\ast}.
\end{align*}

(iv) By applying twice both H\"older's inequality and Lemma \ref{lemmajn}, we get
\begin{align*}
\Vert \transf{\psi}{a} \transf{f\circ\varphi}{a}\Vert_2^2 
&= \Vert \transf{\psi}{a}^2 \transf{f\circ\varphi}{a}^2\Vert_1 \\
&\le \Vert \transf{\psi}{a} \Vert_2 \Vert \transf{\psi}{a} \Vert_4 \Vert \transf{f\circ\varphi}{a}\Vert_8^2\\
&\le \Vert \transf{\psi}{a} \Vert_2 \Vert \psi \Vert_{\ast,4} \Vert f\circ\varphi \Vert_{\ast,8}^2\\
&\lesssim \beta(a) \Vert \psi \Vert_{\ast} \Vert f\circ\varphi \Vert_{\ast}^2/L(\varphi(a)).
\end{align*}
By applying the estimate $\log 2 \le L(\varphi(a))$ and the norm estimate \eqref{eq_norm_bmoa}, we have $\Vert \psi \Vert_{\ast} \le \sup_{w\in\BD}\beta(w)\lesssim \Vert \wco \Vert_{L(BMOA)}$. Hence
\begin{align*}
\beta(a) \Vert \psi \Vert_{\ast} \Vert f\circ\varphi \Vert_{\ast}^2/L(\varphi(a)) &\lesssim
(\sup_{w\in\BD}\beta(w))^2\Vert f \Vert_{\ast}^2/ L(\varphi(a))\\
&\lesssim \Vert f \Vert_{\ast}^2 \min\left\{\sup_{w\in\BD}\beta(w), \frac{\Vert \wco \Vert_{L(BMOA)}}{\sqrt{L(\varphi(a))}}\right\}^2.
\end{align*}
\end{proof}

Finally, we present a weighted version of \cite[Lemma 2.3]{GLL}. The idea of the proof is essentially contained also in \cite{wulan} or \cite[pp.~187-188]{col}. 

\begin{lemma}\label{lemmaupperest2} Suppose that $\wco$ is bounded on BMOA and let $f_a = \sigma_{\varphi(a)}-\varphi(a)$. Then
\begin{align*}
\sup_{a\in\UnitDisk}\Vert \wco  f_a \Vert_{\ast} \le 2\sup_{n\ge 0} \Vert \psi \varphi^n \Vert_{\ast}
\end{align*}
and
\begin{align*}
\limsup_{|\varphi(a)|\to 1}\Vert \wco  f_a \Vert_{\ast} \le 2\limsup_{n\to\infty} \Vert \psi \varphi^n \Vert_{\ast}.
\end{align*}
\end{lemma}

\begin{proof} 
We only prove the second inequality, because the first inequality follows with a similar argument.
Because $\wco f_a = (|\varphi(a)|^2 - 1)\sum_{n=0}^\infty \overline{\varphi(a)}^n \psi\varphi^{n+1}$, we have
$$\|\wco f_a\|_{\ast} \le (1 - |\varphi(a)|^2) \sum_{n=0}^\infty |\varphi(a)|^n \|\psi\varphi^{n+1}\|_{\ast},$$
and for each $N$, that
$$\|\wco f_a\|_{\ast} \le \left(1 - |\varphi(a)|^2\right) \left(\sum_{n=0}^N |\varphi(a)|^n \|\psi\varphi^{n+1}\|_{\ast} + \sum_{n=N+1}^\infty|\varphi(a)|^n \|\psi\varphi^{n+1}\|_{\ast}\right)$$
$$\le \left(1 - |\varphi(a)|^2\right) \left(\sum_{n=0}^N |\varphi(a)|^n \|\psi \varphi^{n+1}\|_{\ast} + \sup_{n\ge N+1} \|\psi\varphi^{n+1}\|_{\ast} \frac{|\varphi(a)|^{N +1}}{1 - |\varphi(a)|}\right).$$
Therefore
\begin{align*}
\limsup_{|\varphi(a)|\to 1} \|\wco f_a\|_{\ast}
&\le 2\sup_{n\ge N+1}\|\psi \varphi^{n+1}\|_{\ast},
\end{align*}
and we complete the proof by letting $N\to\infty$.
\end{proof}


\section{Proof of Theorem \ref{mainthm1}}\label{sec_pfofthm1}

The proof of Theorem \ref{mainthm1} follows immediately from the next two lemmas.

\begin{lemma} If $\wco$ is bounded on BMOA, then
\begin{align*}
\Vert \wco \Vert_{L(BMOA)} \lesssim  \sup_{n\ge 0} \Vert \psi \varphi^n \Vert_{\ast} + \sup_{a\in\BD}\beta(a) + |\psi(0)| L(\varphi(0)).
\end{align*}
\end{lemma}

\begin{proof} 
By Lemma \ref{lemmaupperest} (iii) and (iv), we have
\begin{align*}
\Vert \wco f\Vert_{\ast}  \lesssim 
\sup_{a\in\UnitDisk}\left(\alpha(a) + \beta(a) \right)\Vert f \Vert_{\ast},
\end{align*}
for all $f\in BMOA$.
Hence, by Lemma \ref{lemmaupperest} (i) and Lemma \ref{lemmaupperest2},
\begin{align*}
\alpha(a) \lesssim 
\beta(a) + \sup_{n\ge 0} \Vert \psi \varphi^n \Vert_{\ast},
\end{align*}
because $L(\varphi(a))$ is bounded from below. Thus
\begin{align*}
\Vert \wco f\Vert_{\ast}  \lesssim 
\left(\sup_{n\ge 0} \Vert \psi \varphi^n \Vert_{\ast} + \sup_{a\in\UnitDisk}\beta(a) \right)\Vert f \Vert_{\ast}.
\end{align*}
We further have $|(\wco f)(0)|=|\psi(0) \Vert f(\varphi(0))|\lesssim |\psi(0)| L(\varphi(0)) \Vert f\Vert$, by Lemma \ref{lemmaptwiseie}, which completes the proof.
\end{proof}

\begin{lemma} Suppose that $\wco\colon BMOA \to BMOA$ is bounded. Then
\begin{align*}
\Vert \wco \Vert_{L(BMOA)} \gtrsim  \sup_{n\ge 0} \Vert \psi \varphi^n \Vert_{\ast} + \sup_{a}\beta(a) + |\psi(0)| L(\varphi(0)).
\end{align*}
\end{lemma}

\begin{proof}
Combine the norm estimate \eqref{eq_norm_bmoa} and the fact that the functions $z\mapsto z^n$ belong uniformly to $BMOA$ for $n\ge 0$.
\end{proof}


\section{Proof of Theorem \ref{mainthm2}}\label{sec_pfofthm2}

In this section we prove Theorem \ref{mainthm2}. We split the proof into three lemmas of which Theorem \ref{mainthm2} is a direct consequence.

\begin{lemma} Suppose that $\wco$ is bounded on $BMOA$. Then
\begin{align*}
\Vert \wco \Vert_{e,BMOA} \gtrsim  \limsup_{n\to\infty} \Vert \psi \varphi^n \Vert_{\ast}+ \limsup_{|\varphi(a)|\to 1}\beta(a).
\end{align*}
\end{lemma}

\begin{proof}
A standard argument shows that for any $f_n\in VMOA$ such that $\sup_n\Vert f_n \Vert<\infty$ and $f_n \to 0$ weakly in BMOA as $n\to\infty$, we have
\begin{align*}
\Vert \wco \Vert_{e,BMOA} \gtrsim \limsup_{n\to\infty} \Vert \wco f_n \Vert;
\end{align*}
see, for example, \cite{GLL}. We next apply this basic result to three different sequences of test functions.
First, by setting $f_n(z)=z^n$, we get 
\begin{align}\label{eq_lowerest1}
\Vert \wco \Vert_{e,BMOA} \gtrsim \limsup_{n\to\infty} \Vert \wco z^n \Vert_{\ast}=\limsup_{n\to\infty} \Vert \psi \varphi^n \Vert_{\ast}. 
\end{align}
If now $\Vert \varphi \Vert_\infty < 1$, then $\limsup_{|\varphi(a)|\to 1}\beta(a)=0$ and the proof is complete. Otherwise, there
are points $a_n\in\UnitDisk$ such that $|\varphi(a_n)|\to 1$ and 
\begin{align*}
\lim_n\alpha(a_n) = \limsup_{|\varphi(a)|\to 1}\alpha(a).
\end{align*}
Define $f_n(z)=\sigma_{\varphi(a_n)}-\varphi(a_n)$. Then, by Lemma \ref{lemmaupperest} (i),
\begin{align}\label{eq_lowerest2}
\begin{split}
\Vert \wco \Vert_{e,BMOA} 
&\gtrsim \limsup_{n\to\infty} \Vert \wco f_n \Vert_{\ast}\\
& \gtrsim \limsup_{n\to\infty}\left(\alpha(a_n)-\Vert \wco \Vert_{L(BMOA)}/L(\varphi(a_n))\right)\\
&=\limsup_{|\varphi(a)|\to 1}\alpha(a).
\end{split}
\end{align}
Finally, choose 
$a_n\in\UnitDisk$ such that $|\varphi(a_n)|\to 1$ and 
\begin{align*}
\lim_{n\to\infty}\beta(a_n)=\limsup_{|\varphi(a)|\to 1}\beta(a).
\end{align*}
Let $g_{a_n}$ be as in Lemma \ref{lemmaupperest} (ii). It was shown in \cite[p.~35]{lai2} that the functions $g_{a_n}$ are uniformly bounded in $BMOA$, $g_{a_n} \in VMOA$ and $g_{a_n}$ converges weakly to zero in $BMOA$. By these facts
and the estimates \eqref{eq_lowerest2}, we get
\begin{align*}
\Vert \wco \Vert_{e,BMOA} &\gtrsim \limsup_{n\to\infty} \left(\Vert \wco g_{a_n} \Vert_{\ast} + \alpha(a_n)\right).
\end{align*}
By Lemma \ref{lemmaupperest} (ii) and (iv),
\begin{align}\label{eq_lowerest3}
\begin{split}
\Vert \wco \Vert_{e,BMOA} 
&\gtrsim\limsup_{n\to\infty} \left(\beta(a_n)-\Vert g_{a_n} \Vert_{\ast}\Vert\wco\Vert_{L(BMOA)}/\sqrt{L(\varphi(a_n))}\right)\\
&= \limsup_{n\to\infty} \beta(a_n).
\end{split}
\end{align}
Combining above estimates \eqref{eq_lowerest1} and \eqref{eq_lowerest3} completes the proof.
\end{proof}

In next lemmas we will next use the well-known fact that every $f\in H^p$ has the almost everywhere on $\UnitCircle$ existing radial limit function (also denoted by $f$) and that the $H^p$ norm of $f$ coincides with the $L^p(\mathrm{d\theta}/2\pi)$ norm of $f$ on $\UnitCircle$.
For $t\in (0,1)$ we define the sets 
$$E(\varphi,a,t)=\{\xi\in\UnitCircle \colon |(\sigma_{\varphi(a)}\circ\varphi\circ\sigma_a)(\xi)|>t\}\subset\UnitCircle.$$  

\begin{lemma}\label{lemmaupperest1} We have that
$$\lim_{r\to 1}\,\limsup_{t\to 1} \sup_{|\varphi(a)|\le r}\left(\int_{E(\varphi,a,t)} |\psi(\sigma_a(\e^{i\theta}))|^4 \,
\frac{\mathrm d\theta}{2\pi}\right)^{1/4}\lesssim \limsup_{n\to\infty} \|\psi\varphi^n\|_{\ast}.$$
\end{lemma}

\begin{proof} The triangle inequality yields that
\begin{align*}
\Vert (\psi\circ\sigma_a) & (\varphi^n\circ\sigma_a -\varphi^n(a))\Vert_4 \\
\le
&\Vert (\psi\circ\sigma_a)(\varphi^n\circ\sigma_a) -\psi(a)\varphi^n(a) \Vert_4 +
 \Vert (\psi\circ\sigma_a -\psi(a))\varphi^n(a) \Vert_4.
\end{align*}
Now fix $r \in (0,1)$. Then for every $|\varphi(a)|\le r$, we obtain that
$$ \Vert (\psi\circ\sigma_a)(\varphi^n\circ\sigma_a -\varphi^n(a)) \Vert_4\le \Vert \psi\varphi^n \Vert_{\ast,4} + r^n  \Vert \psi \Vert_{\ast,4}.$$
For $t \in(r,1)$ we get
\begin{align*}
&\limsup_{t\to 1} \sup_{|\varphi(a)|\le r}\int_{ \tilde E(\varphi,a,t)} |\psi(\sigma_a(\e^{i\theta}))|^4 \,\frac{\mathrm d\theta}{2\pi}\\
&=\limsup_{n\to\infty}\limsup_{t\to 1} \sup_{|\varphi(a)|\le r} (t^n - |\varphi(a)|^n)^4 
\int_{ \tilde E(\varphi,a,t)} |\psi(\sigma_a(\e^{i\theta}))|^4 \, \frac{\mathrm d\theta}{2\pi}\\
&\le\limsup_{n\to\infty}\limsup_{t\to 1} \sup_{|\varphi(a)|\le r}
\int_{ \tilde E(\varphi,a,t)}|(\psi\circ\sigma_a)(\e^{i\theta})((\varphi^n\circ\sigma_a)(e^{i\theta}) - \varphi^n(a))|^4 \, \frac{\mathrm d\theta}{2\pi}\\
&\le\limsup_{n\to\infty}\limsup_{t\to 1} \sup_{|\varphi(a)|\le r}
\Vert (\psi\circ\sigma_a)(\varphi^n\circ\sigma_a -\varphi^n(a)) \Vert_4^4,
\end{align*}
where  $\tilde E(\varphi,a,t)=\{\xi\in\partial\BD\colon |\varphi\circ\sigma_a(\xi)| > t\}.$ 
Therefore
$$\limsup_{t\to 1} \sup_{|\varphi(a)|\le r}\int_{ \tilde E(\varphi,a,t)} |\psi(\sigma_a(\e^{i\theta}))|^4 \, \frac{\mathrm d\theta}{2\pi}
\le \limsup_{n\to\infty}\|\psi\varphi^n\|^4_{\ast,4}.$$
By \cite[Remark 3.3]{lai2}, we have 
$$s(r)^{-1}(1 - |\varphi\circ\sigma_a(\xi)|) \le  1 - |\sigma_{\varphi(a)}\circ\varphi\circ\sigma_a(\xi)| \le s(r) (1 - |\varphi\circ\sigma_a(\xi)|),$$
for $|\varphi(a)|\le r$ and $\xi\in\UnitCircle$, where $s(r)=2(1+r)/(1-r)$.
In particular,
\begin{align*}
\tilde E(\varphi,a,1-s(r)^{-1}(1-t)) \subset  E (\varphi,a,t) \subset \tilde E(\varphi,a,1-s(r)(1-t)),
\end{align*}
for $1-s(r)^{-1} < t < 1$.
In view of these estimates, it follows that
\begin{align*}\limsup_{t\to 1}&\sup_{|\varphi(a)|\le r}\int_{ E(\varphi,a,t)} |\psi(\sigma_a(\e^{i\theta}))|^4 \, \frac{\mathrm d\theta}{2\pi}\\
&=\limsup_{t\to 1}\sup_{|\varphi(a)|\le r}
\int_{\tilde E(\varphi,a,t)} |\psi(\sigma_a(\e^{i\theta}))|^4 \, \frac{\mathrm d\theta}{2\pi}\\
&\le \limsup_{n\to\infty}\|\psi\varphi^n\|^4_{\ast,4}.
\end{align*}
We use  the equivalence of the seminorms $\|\cdot\|_{\ast}$ and $\|\cdot\|_{\ast,4}$ (Lemma \ref{lemmajn}) to complete the proof. 
\end{proof}

For $0<r<1$ and $0< t < 1$,  define 
$$Q(r,t)=r\overline\UnitDisk \cup\{\sigma_b(z)\in\UnitDisk\colon b\in r\overline\UnitDisk, z\in t\overline\UnitDisk\}.$$
Thus $Q(r,t)$ is a compact subset of $\UnitDisk$. 

\begin{lemma}Suppose that $\wco$ is bounded on $BMOA$. Then
\begin{align*}
\Vert \wco \Vert_{e,BMOA} \lesssim  \limsup_{n\to\infty} \Vert \psi \varphi^n \Vert_{\ast}+ \limsup_{|\varphi(a)|\to 1}\beta(a).
\end{align*}
\end{lemma}

\begin{proof}
For $n \ge 0$, let $r_n = n/(1+n)$ and define the linear operator $K_n$ on $BMOA$ by $(K_n f)(z)=f(r_n z)$.
Then standard arguments show that $\sup_n\Vert K_n \Vert \le 2$ and every $K_n$ is a compact operator. Moreover, $\sup_{z\in K}\sup_{\Vert f\Vert\le 1}|(S_n f)(z)|$ converges to zero for compact subsets $K\subset \UnitDisk$, where $S_n f = f - K_n f$; see, for example, \cite{GLL}. We now have 
\begin{align}\label{eqfinal1}
\begin{split}
\Vert \wco \Vert_{e,BMOA} &\le \liminf_{n\to\infty}\Vert \wco S_n \Vert \\
&=\liminf_{n\to\infty}\sup_{\Vert f\Vert\le 1}|\psi(0)(S_n f)(\varphi(0)) |\\
&\quad + \liminf_{n\to\infty}\sup_{\Vert f\Vert\le 1}\Vert \wco S_n f \Vert_{\ast}\\
&=\liminf_{n\to\infty}\sup_{\Vert f\Vert\le 1}\Vert \wco S_n f \Vert_{\ast}.
\end{split}
\end{align}

Fix $n\ge 0$, $f \in BMOA$ with $ \Vert f \Vert \le 1$, $r\in (0,1)$ and $t\in (\frac{1}{2},1)$. Then 
\begin{align}\label{eq_claim1}
\Vert \wco S_n f \Vert_{\ast} \le \sup_{|\varphi(a)|\le r}\Vert \transf{\wco S_n f}{a} \Vert_2 +
\sup_{|\varphi(a)|> r}\Vert \transf{\wco S_n f}{a} \Vert_2.
\end{align}
By Lemma \ref{lemmaupperest} (iii) and (iv), we get
\begin{align*}
\sup_{|\varphi(a)|> r}&\Vert \transf{\wco S_n f}{a} \Vert_2\\
&\lesssim \Vert  S_n f\Vert_{\ast}\sup_{|\varphi(a)|> r}\left(\alpha(a)+\beta(a)+\Vert \wco\Vert_{L(BMOA)}/\sqrt{L(\varphi(a))}\right).
\end{align*}
To estimate the other term, we use the triangle inequality and get
\begin{align*}
\sup_{|\varphi(a)|\le r}\Vert \transf{\wco S_n f}{a} \Vert_2 
&\le
\sup_{|\varphi(a)|\le r}|(S_n f)(\varphi(a))|\Vert \transf{\psi}{a} \Vert_2 \\
&\quad+
\sup_{|\varphi(a)|\le r}\Vert \psi\circ\sigma_a \cdot \transf{S_n f\circ\varphi}{a}\Vert_2\\
&\le \Vert \psi \Vert_{\ast}\max_{w\in Q(r,t)}|(S_n f)(w)| + I_1^{1/2} + I_2^{1/2},
\end{align*}
where
\begin{align*}
I_1 &=\sup_{|\varphi(a)|\le r}\int_{\UnitCircle\setminus E(\varphi,a,t)} | (\psi\circ\sigma_a)(\e^{i\theta}) \transf{S_n f\circ\varphi}{a}(\e^{i\theta}) |^2 \,\frac{\mathrm d\theta}{2\pi};\\
I_2 &=\sup_{|\varphi(a)|\le r}\int_{E(\varphi,a,t)} | (\psi\circ\sigma_a)(\e^{i\theta}) \transf{S_n f\circ\varphi}{a}(\e^{i\theta}) |^2 \,\frac{\mathrm d\theta}{2\pi}.
\end{align*}
Abbreviate $\varphi_a=\sigma_{\varphi(a)}\circ\varphi\circ\sigma_a$ and 
note that $\transf{S_n f\circ\varphi}{a} = (S_n f)\circ\sigma_{\varphi(a)}\circ\varphi_a-(S_n f\circ\varphi)(a)$, so that
\begin{align*}
|\transf{S_n f\circ\varphi}{a}(\xi)|\le 2 |\varphi_a(\xi)|\sup_{|w|\le t}| ((S_n f)\circ\sigma_{\varphi(a)})(w)-(S_n f)(\varphi(a))|,
\end{align*}
for $\xi\in \UnitCircle\setminus E(\varphi, a, t)$  by Lemma \ref{lemmaweightedlwie}. Hence
\begin{align*}
I_1 &\le
4\sup_{|\varphi(a)|\le r}\sup_{|w|\le t}| ((S_n f)\circ\sigma_{\varphi(a)})(w)-(S_n f)(\varphi(a))|^2 \Vert \psi\circ\sigma_a \cdot \varphi_a \Vert_2^2\\
&\le 
16\sup_{z\in Q(r,t)}|(S_n f)(z)|^2 \Vert \psi\circ\sigma_a \cdot \varphi_a \Vert_2^2,
\end{align*}
where 
\begin{align*}
\Vert \psi\circ\sigma_a \cdot \varphi_a \Vert_2 &\le \Vert \transf{\psi}{a} \Vert_2 \Vert \varphi_a \Vert_\infty + |\psi(a)|\Vert \varphi_a\Vert_2 \\
& \lesssim \Vert \psi \Vert_{\ast} + \alpha(a) \lesssim \Vert \wco \Vert_{L(BMOA)}.
\end{align*}

For $I_2$ we use  H\"older's inequality;
$$I_2\le \sup_{|\varphi(a)|\le r}\left(\int_{E(\varphi,a,t)} |\psi(\sigma_a(\e^{i\theta}))|^4 \, \frac{\mathrm d\theta}{2\pi}\right)^{1/2}
\sup_{|\varphi(a)|\le r}\Vert \transf{S_n f\circ\varphi}{a} \Vert_4^2,$$
where
$$\Vert \transf{S_n f\circ\varphi}{a} \Vert_4^2\le  \Vert S_nf\circ\varphi \Vert ^2_{\ast,4}\lesssim  \Vert f \Vert ^2_{\ast}\le 1,$$
by Lemma \ref{lemmajn}. 
By combining the above estimates with \eqref{eq_claim1}, we have for $r\in (0,1)$ and $t\in(\frac{1}{2},1)$ that
\begin{align*}
\Vert \wco S_n f\Vert_{\ast} &\lesssim \sup_{|\varphi(a)|> r}\left(\alpha(a)+\beta(a)+\Vert \wco\Vert_{L(BMOA)}/\sqrt{L(\varphi(a))}\right)\\
&\quad + 
\sup_{|\varphi(a)|\le r}\left(\int_{ E(\varphi,a,t)} |\psi(\sigma_a(\e^{i\theta}))|^4 \,
 \frac{\mathrm d\theta}{2\pi}\right)^{1/4}\\
&\quad +
\sup_{z\in Q(r,t)}|(S_n f)(z)|^2 \Vert \wco \Vert_{L(BMOA)}.
\end{align*}
By taking the supremum over $\Vert f \Vert \le 1$, letting $n\to \infty$ and applying \eqref{eqfinal1}, we get
\begin{align*}
\Vert\wco\Vert_{e,BMOA} &\lesssim \sup_{|\varphi(a)|> r}\left(\alpha(a)+\beta(a)+\Vert \wco\Vert_{L(BMOA)}/\sqrt{L(\varphi(a))}\right)\\
&\quad + 
\sup_{|\varphi(a)|\le r}\left(\int_{ E(\varphi,a,t)} |\psi(\sigma_a(\e^{i\theta}))|^4 \,
\frac{\mathrm d\theta}{2\pi}\right)^{1/4}.
\end{align*}
From Lemma \ref{lemmaupperest} (i) and Lemma \ref{lemmaupperest2} we conclude that 
\begin{align*}
\limsup_{|\varphi(a)| \to 1}\alpha(a) \lesssim \limsup_{n\to\infty} \Vert \psi \varphi^n \Vert_{\ast}.
\end{align*}
Therefore, by  combining the above estimates with Lemma \ref{lemmaupperest1}, we see that 
$$ \Vert \wco \Vert_{e,BMOA} \lesssim \limsup_{|\varphi(a)|\to 1}\beta(a)+\limsup_{n\to\infty}  \Vert \psi\varphi^n \Vert_{\ast},$$
and the proof is complete.
\end{proof}

We finally prove Corollary \ref{corend}.

\begin{proof}[Proof of Corollary \ref{corend}] First we note that $VMOA^{**} = BMOA$ and $\wco$ on $BMOA$ is the second adjoint of $\wco$ on $VMOA$. Moreover, $(VMOA)^* = H^1$ has the metric approximation property, so  Theorem 3 in \cite{axler} gives us that the essential norms $ \Vert \wco \Vert_{e,BMOA}$ and $ \Vert \wco \Vert_{e,VMOA}$ are equivalent.
Since $\psi$ and $\psi\varphi$ belong to $VMOA$, we conclude (see Proposition 4.1 in \cite{lai2}) that
$$\limsup_{|a|\to 1}\Vert \psi\circ\sigma_a - \psi(a)\Vert_2=0 \quad \text{and} \quad \limsup_{|a|\to 1} |\psi(a)| \Vert
\varphi\circ\sigma_a - \varphi(a)\Vert_2=0.$$
Let $r\in (0,1)$ be fixed.  Since  
\begin{align*}
\|\sigma_{\varphi(a)}\circ\varphi\circ\sigma_a\|_2 = \left\Vert \frac{\varphi(a)-\varphi\circ\sigma_a}{1-\overline{\varphi(a)}\varphi\circ\sigma_a} \right\Vert_2 \le \frac{\|\varphi\circ\sigma_a - \varphi(a)\|_2}{1 -|\varphi(a)|},
\end{align*}
we get 
\begin{align*}
\lim_{s\to 1} \sup_{|\varphi(a)|\le r, |a|>s}& |\psi(a)| \|\sigma_{\varphi(a)}\circ\varphi\circ\sigma_a\|_2 \\
&\le (1 -r)^{-1} \limsup_{|a|\to 1} |\psi(a)| \|\varphi\circ\sigma_a - \varphi(a)\|_2 = 0
\end{align*}
and 
\begin{align*}
\lim_{s\to 1}\sup_{|\varphi(a)|\le r, |a|>s}& \left(\log\frac{2}{1 - |\varphi(a)|^2}\right)
\Vert \psi\circ\sigma_a - \psi(a)\Vert_2
\\&\le \left(\log \frac{2}{1 -r^2}\right) \limsup_{|a|\to 1}\Vert \psi\circ\sigma_a - \psi(a)\Vert_2=0.
\end{align*}
Hence,
$$\limsup_{|a|\to 1} \alpha(a) \le \lim_{s\to 1} \left( \sup_{|\varphi(a)|\le r, |a|>s} \alpha(a) + 
\sup_{|\varphi(a)| > r, |a|>s}\alpha(a)\right) \le \sup_{|\varphi(a)|>  r} \alpha(a)$$
and, similarly,
$$\limsup_{|a|\to 1} \beta(a)  \le \sup_{|\varphi(a)|>  r} \beta(a).$$
Therefore we can use the essential norm estimate \eqref{eq_ess_norm_vmoa} for $VMOA$ to conclude that 
$$ \Vert \wco \Vert_{e,BMOA}\lesssim \limsup_{|\varphi(a)|\to 1} \alpha(a) + \limsup_{|\varphi(a)|\to 1}\beta(a).$$
On the other hand, 
\begin{align*}
\limsup_{|\varphi(a)| \to 1}\alpha(a) \lesssim \limsup_{n\to\infty} \Vert \psi \varphi^n \Vert_{\ast},
\end{align*}
so Theorem \ref{mainthm2} finishes the proof.
\end{proof}



\begin{thebibliography}{}
%
%

\bibitem{axler} Axler, S., Jewell, N., Shields, A.: The essential norm of an operator and its adjoint, Trans. Amer. Math. Soc. \textbf{261}, 159--167 (1980)

\bibitem{col}  Colonna, F.: Weighted composition operators between $H^\infty$ and BMOA. Bull. Korean Math. Soc. \textbf{50}, 185--200 (2013)

\bibitem{comac} Cowen, C.C., MacCluer, B.D.: Composition operators on spaces of analytic functions. CRC Press, Boca Raton (1995)

\bibitem{GLL} Galindo, P., Laitila, J., Lindstr\"om, M.: Essential norm estimates for composition operators on BMOA. J. Funct. Anal. \textbf{265}, 629--643 (2013)

\bibitem{gar} Garnett, J.B.: Bounded Analytic Functions, Revised First Edition. Springer, New York (2007)

\bibitem{gir} Girela, D.: Analytic functions of bounded mean oscillation. In: Complex Function Spaces
(Mekrij\"arvi, 1999), Univ. Joensuu Dept. Math. Rep. Ser. 4, Univ. Joensuu, Joensuu, pp. 61--170 (2001)

\bibitem {mik1} Hyv\"{a}rinen, O., Kemppainen, M., Lindstr\"{o}m, M., Rautio, A., Saukko, E.: The essential norms of
weighted composition operators on weighted Banach spaces of analytic functions. Integral Equations Operator Theory \textbf{72}, 151--157 (2012)

\bibitem {mik2} Hyv\"{a}rinen, O., Lindstr\"{o}m, M.: Estimates of essential norms of
weighted composition operators between  Bloch type spaces. J. Math. Anal. Appl. \textbf{393}, 38--44 (2012)

\bibitem{lai1} Laitila, J.: Composition operators and vector-valued BMOA. Integral Equations Operator Theory \textbf{58}, 487--502 (2007)

\bibitem{lai2} Laitila, J.: Weighted composition operators on BMOA. Comput. Methods Funct. Theory \textbf{9}, 27--46 (2009)

\bibitem{LNST} Laitila, J., Nieminen, P.N., Saksman, E., Tylli, H.-O.: Compact and weakly compact composition operators on BMOA. Complex Anal. Oper. Theory \textbf{7} 163--181 (2013)

\bibitem {manhas} Manhas, J.S., Zhao, R.: New estimates of essential norms of weighted composition operators between Bloch type spaces. J. Math. Anal. Appl. \textbf{389}, 32--47 (2012)

\bibitem{shapiro1} Shapiro, J.H.: Composition operators and classical function theory. Springer, New York (1993)

\bibitem{smith}  Smith, W.: Compactness of composition operators on BMOA. Proc. Amer. Math. Soc. \textbf{127}, 2715--2725 (1999)

\bibitem{steg}  Stegenga, D.A.: Bounded Toeplitz operators on $H^1$ and applications of the duality between $H^1$ and the functions of bounded mean oscillation. 
Amer. J. Math. \textbf{98}, 573--589 (1976)

\bibitem{wulan1} Wulan, H.: Compactness of composition operators on BMOA and VMOA. Sci. China Ser. A \textbf{50}, 997--1004 (2007)

\bibitem {wulan} Wulan, H., Zheng, D., Zhu, K.: Composition operators on BMOA and the Bloch space. Proc. Amer. Math. Soc. \textbf{137}, 3861--3868 (2009)

\bibitem {zhao}  Zhao, R.: Essential norms of  composition operators between Bloch type spaces. Proc. Amer. Math. Soc. \textbf{138}, 2537--2546 (2010)

\bibitem{zhu} Zhu, K.: Operator Theory in Function Spaces. Marcel Dekker, New York (1990)

\end{thebibliography}
\end{document}